\documentclass[letterpaper,10pt]{amsart}
\usepackage{amsmath,amsfonts,amssymb}

\newtheorem{theorem}{Theorem}[section]

\newtheorem{proposition}[theorem]{Proposition}

\newcommand{\C}{\ensuremath{{\mathbb C}}}

\newcommand{\F}{\ensuremath{{\mathbb F}}}

\newcommand{\p}{\ensuremath{{\mathfrak p}}}
\newcommand{\np}{\ensuremath{{\left|\mathfrak p\right|}}}

\newcommand{\OO}{\ensuremath{{\mathcal O}}}
\newcommand{\OOmon}{\ensuremath{{\mathcal O}_{mon} }}

\newcommand{\tha}{^{\mathrm{th}}}
\newcommand{\prs}[2]{\left(\frac{#1}{#2}\right)}
\newcommand{\sqf}[1]{#1_\flat}

\newcommand{\dsum}[2]{\displaystyle {\sum_{#1}^{#2}\:}}

\DeclareMathOperator{\cond}{cond}

\numberwithin{theorem}{section}
\numberwithin{equation}{section}

\begin{document}

\title{Sums of $L$-functions over the rational function field}
\author{Gautam Chinta and Joel B. Mohler}
 
\address{Department of Mathematics,
The City College of New York,
New York, NY 10031, USA
}
\email{chinta@sci.ccny.cuny.edu}

\address{Department of Mathematics,
Lehigh University,
Bethlehem, PA 18015, USA
}
\email{jbm5@lehigh.edu}

\thanks{The first named author is partially supported by a grant from
  the NSF}
\date {12 June 2007}

\begin{abstract}
Friedberg, Hoffstein and Lieman have  constructed two related 
multiple Dirichlet series from quadratic and higher-order 
$L$-functions and Gauss sums.  We compute these multiple Dirichlet series 
explicitly in the case of the rational function field.  This is done by 
utilizing the functional equation of the $L$-functions and the functional 
equation relating the two multiple Dirichlet series.  We also point out 
a very simple correspondence between these series and their $p$-parts.
\end{abstract}

\maketitle



\section{Introduction}
Let $n$ be an integer greater than or equal to 2, $\F_q$ the finite
field with $q$ elements, and $K=\F_q(t)$ the rational function field.
The main result of this paper is the explicit computation of an
infinite sum of $L$-functions associated to $n^{th}$ order Hecke
characters of $K.$ The infinite sums we consider are examples of
double Dirichlet series in two complex variables, and can be written
as power series in $q^{-s}$ and $q^{-w}.$ In fact it will turn out
that the series we construct will be rational functions in $q^{-s}$
and $q^{-w}.$

These series are function field analogues of the series studied by
Friedberg, Hoffstein and Lieman in \cite{fhl:mds}.  In that paper, working
over a number field $F$ containing the $n^{th}$ roots of unity, the
  authors study a double Dirichlet series which is roughly of the form
  \begin{equation*}
    \sum_m L(s, \chi_m) (Nm)^{-w},
  \end{equation*}
where the sum is over integral ideals $m$ of $F$,  
$\chi_m$ is the $n^{th}$ order 
power residue symbol associated to $m$ and  $Nm$ denotes the absolute
norm.  The authors show that this double Dirichlet series has a
meromorphic continuation to all $(s,w)\in \C^2$ and satisfies  a group of
functional equations relating it to a second series constructed from
Gauss sums.  The main ingredients in the proof are the functional
equation of $L(s,\chi_m)$, properties of the Fourier coefficients of
the metaplectic Eisenstein series on the $n$-fold cover of $GL_2$ and
Bochner's tube theorem.

In the case $n=2$, these ideas were
applied by Fisher and Friedberg \cite{ff:mds} in the context of a general
function field to show the rationality of double Dirichlet series
constructed from quadratic $L$-functions.  The case $n=2$ is somewhat
easier because the Gauss sum arising in the functional equation of a
quadratic Hecke $L$-series is trivial, and the theory of metaplectic
Eisenstein series is not needed.

In this paper, we follow a more elementary method originally introduced in
\cite{cfh-survey} in the case $n=2$.  We exploit the fact that 
\begin{equation*}
  \sum_{\substack{d\in \F_q[t]\\ \deg d=k}} \prs{d}{m}
\end{equation*}
vanishes if $k$ is bigger than or equal to the degree of $m$, unless $m$ is a
perfect $n^{th}$ power.  Here $\prs{d}{m}=\chi_m(d)$ denotes the
$n^{th}$ power residue symbol for $m, d$ relatively prime.
If $m$ and $d$ are monic, then we have the reciprocity law
\begin{equation}
  \label{eq:1}
  \prs{m}{d}=\prs{d}{m},
\end{equation}
when $q$ is congruent to $1$ mod $2n,$
see e.g. Rosen \cite{rosen:function_fields}, Theorem 3.5.

We now describe our results more precisely.  We will define two
double Dirichlet series, explicitly compute them as rational functions in 
$q^{-s}, q^{-w}$ and show that they satisfy functional equations which 
relate them to one another.  We begin by 
defining  two multiplicative weighting factors $a(d,m)$ and $b(d,m)$ 
for pairs of monic polynomials, as in \cite{fhl:mds}.  For a monic
prime polynomial $\p,$ let 
\begin{equation}
  \label{eq:a}
a(\p^j,\p^k)=
\begin{cases}
\np^{(n-1)d/n}&\textrm{if }d=\textrm{min}(j,k)\textrm{ and }d\equiv 0 \textrm{ mod }n\\
0&\textrm{otherwise}
\end{cases}
\end{equation}
and 
\begin{equation}
  \label{eq:b}
b(\p^j,\p^k)=
\begin{cases}
1&\textrm{if }k=0\\
\np^{k/2-1}(\np-1) &\textrm{if } j\geq k, k\equiv 0\mbox{\ mod }n,k>0\\
-\np^{k/2-1} &\textrm{if }j=k-1, k\equiv 0\mbox{\ mod }n,k>0\\
\np^{(k-1)/2} &\textrm{if }j=k-1, k\not\equiv 0\mbox{\ mod }n,k>0\\
0&\textrm{otherwise}
\end{cases}
\end{equation}
Then define
\[ a(d,m)=\prod_{\substack{\p^j||d \\ \p^k||m}} a(\p^j,\p^k),\ \ 
b(d,m)=\prod_{\substack{\p^j||d \\ \p^k||m}} b(\p^j,\p^k). \]
Here, $|d|$ denotes the norm $q^{\deg d}.$

Let $\OO$ denote $\F_q[t]$, $\OOmon$ the set of monic
polynomials in $\F_q[t]$ and let $\zeta_\OO(s)$ be the zeta function
of the ring $\OO$:
$$\zeta_\OO(s)=(1-q^{1-s})^{-1}.$$ 
The first double Dirichlet series we consider is 
\begin{equation}
  \label{eq:2}
Z_1(s,w)=\sum_{\substack{d,m\in\OOmon}}
\frac{\chi_{m_0}(\hat{d})a(d,m)}{|m|^w|d|^s}\textrm{,}  
\end{equation}
where $m_0$ is the $n^{th}$ powerfree part of
$m$ and $\hat d$ is the part of $d$ relatively prime to $m_0.$
We show in the following section that this can be rewritten
in terms of $L$-functions:
\begin{equation}
  \label{eq:4}
  Z_1(s,w)=\sum_{\substack{m\in\OOmon}}
\frac{L(s,\chi_{m_0})}{|m|^w}P(s;m)
\end{equation}
where the $P(s;m)$ are finite Euler products defined in 
Proposition \ref{prop:Lhatchi}.

The second multiple Dirichlet series is built out of Gauss sums.  See
the following section for the precise definition of the Gauss sum
$g(r,\epsilon,\chi_m).$
Then
\begin{equation}
\label{z2def}
Z_2(s,w)
=\zeta_{\OO}(nw-\frac{n}{2}+1)\sum_{\substack{d,m\in\OOmon}}
\frac{g(1,\epsilon,\chi_{m_0})}{\sqrt{|\sqf{m}|}}\frac{\bar{\chi}_{m_0}(\hat{d}) b(d,m)}{|m|^w|d|^s}
\end{equation}
where $\sqf{m}$ is the squarefree part of the $n\tha$ powerfree part of $m$. 

We can now state our main theorems.  The first describes a set of functional equations relating 
$Z_1$ and $Z_2$.  Specifically, define 
\[
Z_1(s,w;\delta_i)=
\sum_{\substack{d,m\in \OOmon\\\deg m\equiv i\ (\text{mod\ } n)}}
\frac{\chi_{m_0}(\hat{d})a(d,m)}{|m|^w|d|^s}
\]
and 
\[
Z_2(s,w;\delta_i)=\zeta_{\OO}(nw-\frac{n}{2}+1)
\sum_{\substack{d,m\in \OOmon\\\deg m\equiv i\ (\text{mod\ } n)}}
\frac{g(1,\epsilon,\chi_{m_0})}{\sqrt{|\sqf{m}|}}\frac{\bar{\chi}_{m_0}(\hat{d}) b(d,m)}{|m|^w|d|^s}\mbox{.}
\]
\begin{theorem}\label{thm:fe}
We have the functional equation
\[
Z_1(s,w;\delta_i)=\left\{\begin{array}{ll}
q^{2s-1}\frac{1-q^{-s}}{1-q^{s-1}} Z_2(1-s,w+s-\frac{1}{2};\delta_0)
&  \text{for $i=0$ }\\
q^{2s-1}q^{1/2-s}\frac{\bar{\tau}(\epsilon^i)}{\sqrt{q}} 
Z_2(1-s,w+s-\frac{1}{2};\delta_i) & \text{for $0<i<n$. }
\end{array}\right.
\]
The finite field Gauss sum $\tau(\epsilon^i)$ is defined in the
following section. 
\end{theorem}
This is proved in Section \ref{section_z_i_fe}.

The second main theorem is 
\begin{theorem}\label{thm:Z1Z2}
The double Dirichlet series $Z_1$ and $Z_2$ are rational functions of
$x=q^{-s}$ and $y=q^{-w}$.  Explicitly,
\begin{equation}\label{eqn:z1rf}
Z_1(s,w)=\frac{1-q^2xy}{(1-qx)(1-qy)(1-q^{n+1}x^ny^n)}\textrm{,}
\end{equation}
and
\begin{equation}\label{eqn:z2rf}
 Z_2(s,w)=\frac{1-q^{3n/2}x^{n-1}y^n+
\sum_{i=1}^{n-1}\left(\tau(\epsilon^i) q^{i-1+i/2}x^{i-1}y^i
-\tau(\epsilon^i) q^{3i/2}x^iy^i\right)}
{(1-qx)(1-q^{n/2+1}y^n)(1-q^{3n/2}x^ny^n)}\mbox{.}
\end{equation}
\end{theorem}
This theorem is proved in Section \ref{section_eval}.

We conclude this introduction with some remarks to put our results in
the larger context of Weyl group multiple Dirichlet series.  The
general theory of Weyl group multiple Dirichlet series was introduced
in \cite{wmd1} in order to unify and extend several constructions
which had previously been studied.  In particular the series $Z_1$ and
$Z_2$ of the present paper are expected to be $n-2$ fold residues of
the $n$ variable Weyl group multiple Dirichlet series associated to
the root system $A_n$. Brubaker and Bump \cite{bb:wmd-fhl} have verified this
in the cubic case $n=3.$ This case is manageable because, thanks to
Patterson \cite{pa:cubic1,pa:cubic2}, we have a complete understanding of the
Fourier coefficients of the theta function on the 3-fold metaplectic
cover of $SL_2$ which arise
when we take the residues of the Dirichlet series constructed from
cubic Gauss sums.  For $n>3$ the precise nature of the coefficients of the
$n$-fold cover theta functions remains mysterious.  Nevertheless, there is
much evidence in favor of the expectation that the two series
constructed by Friedberg, Hoffstein and Lieman coincide with a
multiresidue of a Weyl group multiple Dirichlet series.  Indeed, one
of the motivations of this paper is to lay the groundwork for
investigating this question for $n>3$ by explicitly computing and
comparing the the relevant multiple Dirichlet series in the case of
the rational functional field.  For example, in \cite{ffwmd} the first
named author has explicitly computed the cubic $A_3$ multiple
Dirichlet series and checked that residues of this series give the two
series in Theorem \ref{thm:Z1Z2} of this paper when  $n=3.$

Finally we point out a curious connection between the series $Z_1, Z_2$ of
Theorem \ref{thm:Z1Z2} and their $\p$-parts.  Define the following
generating series $H_1,H_2$ constructed from the respective $\p$-parts
of $Z_1$ and $Z_2,$
\begin{equation}\begin{split}
  \label{eq:5}
  H_1(X,Y)&=\sum_{j,k\geq 0} a(\p^j,\p^k) X^jY^k,
 \\
  H_2(X,Y)&=(1-\np^{n/2-1}Y^{n})^{-1}
\sum_{j,k\geq 0} b(\p^j,\p^k)
    \frac{g(1,\epsilon,\chi_{\p^k})}{\sqrt{|\sqf{\p^{k}}|}} X^jY^k.
\end{split}\end{equation}
where $X=|\p|^{-s}, Y=|\p|^{-w}.$  Then we will show that
\begin{equation}
  \label{eq:6}
  \begin{split}
H_1(X,Y)&=  \frac{1-XY}{(1-X)(1-Y)(1-|\p|^{n-1}X^nY^n)}, \\
H_2(X,Y)&=\frac{1-\np^{n/2-1}X^{(n-1)}Y^{n}+\sum_{i=1}^{n-1}\frac{g(1,\epsilon^i,\chi_{\p})}{\sqrt{\np}}X^{(i-1)}Y^i\np^{(i-1)/2}(1-X)}{(1-X)(1-\np^{n/2-1}Y^n)(1-\np^{n/2}X^nY^n)}
\mbox{.}
  \end{split}
\end{equation}

Now note that the substitutions 
\begin{align*}
X & \longrightarrow  q x\\
 Y & \longrightarrow  q y\\
 \np & \longrightarrow  1/q\\
 g(1,\epsilon^i,\mathfrak{p})&\longrightarrow  \tau(\epsilon^i)
\end{align*}
transform $H_i$ into $Z_i$ for $i=1,2$.  This similarity between a 
rational function field multiple Dirichlet series and its $\p$-part 
seems to hold in a much wider context, see e.g. \cite{cfh-survey, ffwmd}.

\section{Gauss sums and $L$-Functions}

In this section we will define the Gauss sums and $L$-functions that
are the constituents of our double Dirichlet series.
We will mostly follow the notation of Patterson \cite{pa:ff1} but with
some adjustments to facilitate comparison with \cite{fhl:mds}.

As in the introduction,  $K$ is the rational
function field $\F_q(t)$ with polynomial ring $\OO=\F_q[t].$  We let
$\OOmon$ denote the subset of $\OO$ consisting of monic polynomials
and let
$K_\infty=\F_q((t))$ denote the field of Laurent series in $t^{-1}.$
Let $\mu_n=\{a\in \F_q:a^n=1\}$ and let $\chi:\F_q^\times\to\mu_n$ be
the character $a\mapsto a^{\frac {q-1}{n}}.$ 

In order to define Gauss sums we first need  an additive character on
$K_\infty.$  Let $e_0$ 
be a nontrivial additive character on the prime field $\F$ of $\F_q.$  
Use this to define a
character $e_\star$ of $\F_q$ by $e_\star(a)=e_0(\text{Tr}_{\F_q/\F}
a).$  Let $\omega$ be the global differential $dx/x^2.$  Finally
define the character $e$ of $K_\infty$ by
$e(y)=e_\star(\text{Res}_\infty (\omega y))$ for $y\in K_\infty.$
Note that
$$\{y \in K:e|y\OO=1\}=\OO.$$

For any $c\in \OO$, we'll use $c_0$ to indicate the $n\tha$-power
free part of $c$ and $\sqf{c}$ for the squarefree part of $c_0$. 
Fix an embedding $\epsilon$ from the the $n^{th}$ roots of unity of
$\F_q$ to $\C^\times$.   For $r,c\in \OO$ we define the Gauss sum
$$g(r,\epsilon, \chi_c)=\sum_{y \mod \sqf{c}} \epsilon\left(\prs{y}{c}\right)
e\left(\frac {ry}{\sqf{c}}\right).$$
We also need the Gauss sums associated to the finite field $\F_q.$
These are defined by 
$$\tau(\epsilon)=\sum_{j\in\mathbb{F}_q}
\epsilon\left(j^{(q-1)/n}\right)e_0\left(j\right)\mbox{.}$$ 

We define the $L$-function associated to $\chi_m$ by
\begin{equation}
\label{lfunction}
L(s,\chi_m)=\sum_{\substack{d\in\OOmon}}\chi_m(d)|d|^{-s}.
\end{equation}
When $m$ is $n\tha$-power free, the $L$-function satisfies a
functional equation which  we will describe now.  
Denote the conductor of the character $\chi_m$ by $\cond \chi_m$. 
Thus 
\[
|\cond\chi_m|=\begin{cases}
|\sqf{m}|&\deg m\equiv 0 (n)\\
q|\sqf{m}|&\deg m\not\equiv 0 (n).\\
                   \end{cases}
\]
Then the completed $L$-function  
\begin{equation}
L^*(s,\chi_m)=\begin{cases}
\frac{1}{1-q^{-s}}L(s,\chi_m)&\deg m\equiv 0 (n)\\
L(s,\chi_m)&\deg m\not\equiv 0 (n).\\
              \end{cases}
\end{equation}
satisfies the functional equation 
\begin{equation}
\label{fe_complete}
L^*(s,\chi_m)=q^{2s-1}|\cond \chi_m|^{1/2-s}\frac{g^*(1,\epsilon,\chi_m)}{|\cond \chi_m|^{1/2}}L^*(1-s,\bar{\chi}_m)\mbox{}
\end{equation}
where
\[
g^*(1,\epsilon,\chi_m)=\begin{cases}
                      g(1,\epsilon,\chi_m)&\deg m\equiv 0\mbox{ (mod n)}\\
                      \bar{\tau}(\epsilon^i) g(1,\epsilon,\chi_m)&\deg
                      m\equiv i\not\equiv 0\mbox{ (mod n)}. 
                     \end{cases} 
\]
>From the functional equation, we see that $L(s,\chi_m)$ is a
polynomial in $q^{-s}$  whose degree is one less than 
the degree of $\sqf m, $ if
$m$ is not a perfect $n\tha$ power. If $m=1,$ we recover the zeta
function
\begin{equation}
  \label{eq:7}
  \zeta_\OO(s)=\sum_{d\in\OOmon} |d|^{-s}=\frac{1}{1-q^{1-s}}.
\end{equation}

Expanding the components at infinity, we have the following functional equations
when $m$ is $n\tha$-power free,   $\deg m\equiv i\mbox{ (mod $n$)}$:
\begin{equation}
\label{fe_incomplete}
L(s, \chi_m)=\begin{cases}
	q^{2s-1}|\sqf{m}|^{1/2-s}\frac{g(1,\epsilon,\chi_m)}{|\sqf{m}|^{1/2}}\frac{1-q^{-s}}{1-q^{-(1-s)}}L(1-s,\bar{\chi}_m)&i=0\\
	q^{2s-1}(q|\sqf{m}|)^{1/2-s}\frac{\bar{\tau}(\epsilon^i)}{\sqrt{q}}\frac{g(1,\epsilon,\chi_m)}{|\sqf{m}|^{1/2}}L(1-s,\bar{\chi}_m)&0<i<n
             \end{cases}\textrm{.}
\end{equation}
This functional equation will be used in Section \ref{section_z_i_fe}
to relate $Z_1$ and $Z_2$.

We now introduce a modified $L$-function related to (\ref{lfunction}) 
by inserting the weighting factor $a(d,m)$.  Define 
\begin{equation}
\label{lfunction2}
L(s,\hat{\chi}_m)=
\sum_{\substack{d\in\OOmon}}\frac{\chi_{m_0}(\hat{d})a(d,m)}{|d|^s}\textrm{.}
\end{equation}
where $\hat{d}$ is the part of $d$ relatively prime to $m_0$.  
Since the weighting function is multiplicative, $L(s, \hat{\chi}_m)$
retains an Euler product,
\[
L(s,\hat{\chi}_m)=\prod_{\substack{\p\in\OOmon\\\textrm{irreducible}}}
(1
+\frac{\chi_{m_0}(\hat \p)a(\p,m)}{|\p|^s}
+\frac{\chi_{m_0}(\hat \p^2)a(\p^2,m)}{|\p|^{2s}}+\ldots)\textrm{.}
\]
Further, since $a(d,m)=1$ when $d$ and $m$ are coprime, this Euler 
product agrees with the original $L$-function Euler product for all but finitely many 
places.

We will relate this modified $L$-function
$L(s,\hat\chi_m)$ to $L(s,\chi_{m_0})$ and derive a bound on its degree as a
polynomial in $q^{-s}$, as long as $m$ is not a perfect $n\tha$-power.  
These properties are given in the following Proposition.

\begin{proposition}
  \label{prop:Lhatchi}
We have
\begin{equation*}
  L(s,\hat\chi_m)=L(s,\chi_{m_0}) P(s;m)
\end{equation*}
where $P(s;m)=\prod_\p P_\p(s;m)$ and $ P_\p(s;m)=$
\begin{equation*}
 \begin{cases}
      \left(1-\chi_{m_0}(\p)\np^{-s}\right)
\dsum{k=0}{n\alpha-1}
      \frac{\chi_{m_0}(\p^k)a(\p^{n\alpha},\p^k)}{\np^{ks}}+\np^{-n\alpha
        s}\np^{(n-1)\alpha} & \text{if $\p\nmid m_0$,}\\
\dsum{k=0}{n\alpha} \frac{a(\p^{n\alpha+i},\p^k)}{\np^{ks}}
                    & \text{if $\p^i|| m_0, i\neq 0$.}\\
 \end{cases}
\end{equation*}
Here $\alpha$ and $i$ are the unique integers with $0\leq i<n$ and $\p^{n\alpha+i}\|m$.
In particular, for $m$ not a perfect $n\tha$ power, the degree of
$L(s,\hat\chi_m)$ as a polynomial in $q^{-s}$ is less than the degree
of $m.$ 
\end{proposition}

\begin{proof}
Begin with the Euler product
\begin{align*}
L(s,\hat{\chi}_m)
=&\prod_\p\sum_{k=0}^\infty \frac{\chi_{m_0}(\hat{\p}^k)a(m,\p^k)}{\np^{ks}}\\
=&\prod_{\substack{\p^{n\alpha}\mid\mid m}}
	\sum_{k=0}^\infty \frac{\chi_{m_0}(\p^k)a(\p^{n\alpha},\p^k)}{\np^{ks}}	\cdot\prod_{\substack{\p^{n\alpha+i}\mid\mid m\\0<i<n}}
	\sum_{k=0}^\infty \frac{a(\p^{n\alpha+i},\p^k)}{\np^{ks}}\\
\end{align*}
For primes $\p$ with $i=0$---that is $\p\nmid m_0$ and
$\p^{n\alpha}||m$, say--- it follows from (\ref{eq:a}) that  the tails of the sum are a 
geometric series with common ratio $\chi_{m_0}(\p)\np^{-s}$.  Thus for
such $\p$ the $\p$ part is 
\begin{equation*}
\sum_{k=0}^{n\alpha-1} \frac{\chi_{m_0}(\p^k)a(\p^{n\alpha},\p^k)}{\np^{ks}}+\frac{\np^{-n\alpha s}\np^{(n-1)\alpha}}{1-\chi_{m_0}(\p)\np^{-s}}
=\left(1-\chi_{m_0}(\p)\np^{-s}\right)^{-1}P_{\p}(s;m)
\end{equation*}
where 
\begin{equation}
  \label{eq:3}
P_{\p}(s;m)=\sum_{k=0}^{n\alpha-1} \frac{\chi_{m_0}(\p^k)a(\p^{n\alpha},\p^k)}{\np^{ks}}\left(1-\chi_{m_0}(\p)\np^{-s}\right)+\np^{-n\alpha s}\np^{(n-1)\alpha}\mbox{.}  
\end{equation}

For primes such that $\p^i||m_0$ with $0<i<n$,
it follows from (\ref{eq:a}) that $a(\p^{n\alpha+i}, \p^k)=0$ for
$k>n\alpha,$ so the $\p$ part is a 
finite sum:
\begin{equation*}
P_{\p}(s;m)= \sum_{k=0}^{n\alpha}
\frac{a(\p^{n\alpha+i},\p^k)}{\np^{ks}}\mbox{.} 
\end{equation*}

Thus 
\begin{equation*}
  L(s,\hat\chi_m)=L(s,\chi_{m_0}) P(s;m)
\end{equation*}
as claimed.
The bound on the degree of $L(s,\hat\chi_m)$ follows from the bound on
the degree of $L(s,\chi_{m_0})$ for $m_0\neq 1$ and the degrees of the 
$P_p(s;m).$
\end{proof}


\section{Functional Equation:  $H_1\to H_2$}
Recall that the generating series $H_1(X,Y)$ and $H_2(X,Y)$ of
(\ref{eq:5}) define the
$p$-parts of $Z_1$ and $Z_2$, respectively.   We describe the
functional equations relating $H_1(X,Y)$ and $H_2(X,Y).$   These will
be used to prove the global functional equation relating $Z_1$ to
$Z_2.$

The functional equations are a direct consequence of the following
Proposition. 
\begin{proposition}\label{prop:h1h2}
The generating series $H_1(X,Y)$ and $H_2(X,Y)$ are rational functions
of $X$ and $Y.$  Explicitly
\begin{equation}\label{eq:10}
  H_1(X,Y)=  \frac{1-XY}{(1-X)(1-Y)(1-|\p|^{n-1}X^nY^n)},
\end{equation}
and
\begin{equation}\label{eq_h2_rat}
H_2(X,Y)=\frac{1-\np^{n/2-1}X^{(n-1)}Y^{n}+\sum_{i=1}^{n-1}\frac{g(1,\epsilon^i,\chi_{\p})}{\sqrt{\np}}X^{(i-1)}Y^i\np^{(i-1)/2}(1-X)}{(1-X)(1-\np^{n/2-1}Y^n)(1-\np^{n/2}X^nY^n)}
\end{equation}
\end{proposition}

\begin{proof}
Equation (\ref{eq:10}) is obvious from the definition (\ref{eq:a}) 
of the $a(\p^k,\p^l).$
The evaluation of $H_2(X,Y)$ is simply a matter of recognizing 
geometric series.  From the definitions of
$b(\p^j,\p^k)$ in \eqref{eq:b} and $H_2$ in \eqref{eq:5}, we have 
\begin{align*}
(1-\np^{n/2-1}Y^n)H_2(X,Y)
=&\sum_{j=0}^\infty \frac{g(1,\epsilon,\chi_{\p^0})}{\sqrt{|\sqf{\p^{0}}|}}X^jY^0\\
+&\sum_{\alpha=1}^\infty\sum_{j=n\alpha}^\infty \np^{n\alpha/2-1}(\np-1)\frac{g(1,\epsilon,\chi_{\p^{n\alpha}})}{\sqrt{|\sqf{\p^{n\alpha}}|}}X^jY^{n\alpha}\\
+&\sum_{\alpha=1}^\infty-\np^{n\alpha/2-1}\frac{g(1,\epsilon,\chi_{\p^{n\alpha}})}{\sqrt{|\sqf{\p^{n\alpha}}|}}X^{n\alpha-1}Y^{n\alpha}\\
+&\sum_{\alpha=0}^\infty\sum_{i=1}^{n-1}\np^{(n\alpha+i-1)/2}\frac{g(1,\epsilon,\chi_{\p^{n\alpha+i}})}{\sqrt{|\sqf{\p^{n\alpha+i}}|}}X^{n\alpha+i-1}Y^{n\alpha+i}
\end{align*}
Evaluating the geometric series
\begin{align*}
(1-\np^{n/2-1}Y^n)H_2(X,Y)
=&\frac{1}{1-X}+\frac{\np^{n/2-1}(\np-1)X^nY^n}{(1-X)(1-\np^{n/2}X^nY^n)}\\
+&\frac{-\np^{n/2-1}X^{n-1}Y^n}{1-\np^{n/2}X^nY^n}+\sum_{i=1}^{n-1}\frac{\frac{g(1,\epsilon^i,\chi_\p)}{\sqrt{\np}}\np^{(i-1)/2}X^{i-1}Y^i}{(1-\np^{n/2}X^nY^n)}\mbox{.}
\end{align*}
Equation (\ref{eq_h2_rat}) follows by combining the four summands above.
\end{proof}

For $0\leq i<n$, define 
\begin{equation}\begin{split}
  \label{p_part_i}
  H_1(X,Y;\delta_i)&=\sum_{\substack{j,k\geq 0\\k\equiv i (n)}} a(\p^j,\p^k) X^jY^k,
 \\
  H_2(X,Y;\delta_i)&=(1-\np^{n/2-1}Y^{n})^{-1}
\sum_{\substack{j,k\geq 0\\k\equiv i (n)}} b(\p^j,\p^k)
    \frac{g(1,\epsilon,\chi_{\p^k})}{\sqrt{|\sqf{\p^{k}}|}}
    \bar\chi_{\p^k}(\hat \p^j)X^jY^k.
\end{split}\end{equation}
We have shown in Proposition \ref{prop:h1h2}
that $H_1$ and $H_2$ are both rational functions in
$\np^{-s}$ and $\np^{-w}$ and it is clear from this Proposition that 
\begin{equation}\label{eqn:h1i}
H_1(X,Y;\delta_i)=\begin{cases}
                   \frac{1-XY^n}{(1-X)(1-Y^n)(1-\np^{n-1}X^nY^n)} & i\equiv 0\\
                   \frac{(1-X)Y^i}{(1-X)(1-Y^n)(1-\np^{n-1}X^nY^n)} & i\not\equiv 0\\
                  \end{cases}
\end{equation}
and
\begin{equation}\label{eqn:h2i}
H_2(X,Y;\delta_i)=\begin{cases}
                   \frac{1-\np^{n/2-1}X^{(n-1)}Y^{n}}{(1-X)(1-\np^{n/2-1}Y^n)(1-\np^{n/2}X^nY^n)} & i\equiv 0\\
                   \frac{\frac{g(1,\epsilon^i,\chi_{\p})}{\sqrt{\np}}X^{(i-1)}Y^i\np^{(i-1)/2}(1-X)}{(1-X)(1-\np^{n/2-1}Y^n)(1-\np^{n/2}X^nY^n)} & i\not\equiv 0\\
                  \end{cases}
\end{equation}

The following theorem establishes a functional equation relating $H_1$ to $H_2$.
\begin{theorem}
\label{thm:h_i_fe}
We have the functional equation 
\[
H_1(\p^{-s},\p^{-w};\delta_i)=
\begin{cases}
\frac{1-\np^{-(1-s)}}{1-\np^{-s}}H_2(\p^{-(1-s)},\p^{-(w+s-1/2)};\delta_0) & i=0 \\
\frac{\sqrt{\np}}{g(1,\epsilon^i,\chi_{\p})}\np^{s-1/2}H_2(\p^{-(1-s)},\p^{-(w+s-1/2)};\delta_i) & 0<i<n\\
\end{cases}
\]
\end{theorem}

\begin{proof}
The proof is by a direct computation using Equations (\ref{eqn:h1i})
and (\ref{eqn:h2i}).
\end{proof}

\section{Functional Equation: $Z_1\to Z_2$}

\label{section_z_i_fe}

There is a set of functional equations relating $Z_1$ and $Z_2$.  These will be described 
in this section.  Define
\begin{equation}
\label{Q_P_fe}
Q(s;m)=\frac{P(1-s;m)}{(m/\sqf{m})^{s-1/2}}\mbox{.}
\end{equation}
With the expansion of $P$ as an Euler product, we see that $Q$ is also an Euler product 
supported in the primes dividing $m$.
\begin{align*}
Q(s;m)
=&\prod_{\substack{\p^{n\alpha+i}\mid\mid m\\i=0}}\frac{1}{\np^{n\alpha(s-1/2)}}P_{\p}(1-s;m)
\cdot\prod_{\substack{\p^{n\alpha+i}\mid\mid m\\0<i<n}}\frac{1}{\np^{(n\alpha+i-1)(s-1/2)}}P_{\p}(1-s;m)\\
\end{align*}

\begin{proposition} Define 
\[
Z_2'(s,w)=\sum_{m}\frac{g(1,\epsilon,\chi_{m_0})}{\sqrt{|\sqf{m}|}}\frac{L(s,\bar{\chi}_{m_0})Q(s;m)}{|m|^w}\mbox{.}
\]
We have $Z_2'=Z_2$.
\end{proposition}

\begin{proof}
Define
\begin{equation}
\label{h2def}
H_2'(\p^{-s},\p^{-w};\delta_i)=\begin{cases}
                \left(1-\p^{-s}\right)^{-1}\sum_{k=0}^{\infty}\frac{Q(s;\p^{nk})}{\p^{nkw}} & i=0\\
                \frac{g(1,\epsilon^i,\chi_{\p})}{\sqrt{\np}}\sum_{k=0}^{\infty}\frac{Q(s;\p^{nk+i})}{\p^{nkw}} & 0<i<n\\
                               \end{cases}\mbox{.}
\end{equation}
Then $H_2'(\p^{-s},\p^{-w})=\sum_{i=0}^{n-1}H_2'(\p^{-s},\p^{-w};\delta_i)$ is the $\p$-part of $Z_2'$.
We will show that $H_2'$ and $H_2$ both satisfy the functional equations with $H_1$ shown in 
theorem \ref{thm:h_i_fe} and therefore $H_2'=H_2$.
The result follows since, for fixed $m$, the $L$-functions, $P$, and
$Q$ all have Euler products.  

As a result of the definition in equation \eqref{Q_P_fe}, the
$\p$-parts of $Q$ and  
$P$ satisfy
\[
P(s;\p^{n\alpha+i})=\begin{cases}
                           \p^{n\alpha(1/2-s)}Q(1-s;\p^{n\alpha})&i=0\\
                           \p^{(n\alpha+i-1)(1/2-s)}Q(1-s;\p^{n\alpha+i})&0<i<n\mbox{.}
                          \end{cases}
\]

Therefore, we relate $H_1$ to $H_2'$ by 
\begin{align*}
H_1(\p^{-s},\p^{-w};\delta_0)
=&(1-\np^{-s})^{-1}\sum_{k=0}^\infty \frac{P(s;\p^{nk})}{\np^{nkw}}\\
=&(1-\np^{-s})^{-1}\sum_{k=0}^\infty \frac{Q(1-s;\p^{nk})\np^{nk(1/2-s)}}{\np^{nkw}}\\
=&(1-\np^{-s})^{-1}\sum_{k=0}^\infty \frac{Q(1-s;\p^{nk})}{\np^{nk(w+s-1/2)}}\\
=&\frac{1-\np^{-(1-s)}}{1-\np^{-s}}H_2'(\p^{-(1-s)},\p^{-(w+s-1/2)};\delta_0)\mbox{.}
\end{align*}
This is exactly the functional equation satisfied by $H_2$ in Theorem \ref{thm:h_i_fe}.  A similar computation 
shows that 
\[
H_1(\p^{-s},\p^{-w};\delta_i)=\frac{\sqrt{\np}}{g(1,\epsilon^i,\chi_{\p})}\np^{s-1/2}H_2(\p^{-(1-s)},\p^{-(w+s-1/2)};\delta_i)
\]
for $0<i<n$.  Thus, $H_2'(X,Y;\delta_i)=H_2(X,Y;\delta_i)$ for all $0\leq i<n$ and this completes the proof.
\end{proof}

For $0\leq i<n$, define 
\[
Z_1(s,w;\delta_i)=\sum_{\substack{m\in \OOmon\\\deg m\equiv i\ 
(\text{mod\ } n)
}}\frac{L(s,\chi_{m_0})P(s;m)}{|m|^w}
\]
and
\[
Z_2(s,w;\delta_i)=\sum_{\substack{m\in \OOmon\\ \deg m\equiv i\ 
(\text{mod\ } n) }}\frac{g(1,\epsilon,\chi_{m_0})}{\sqrt{|\sqf{m}|}}\frac{L(s,\bar{\chi}_{m_0})Q(s;m)}{|m|^w}\mbox{.}
\]

\begin{theorem}
We have the functional equation
\[
Z_1(s,w;\delta_i)=\left\{\begin{array}{ll}
q^{2s-1}\frac{1-q^{-s}}{1-q^{s-1}} Z_2(1-s,w+s-\frac{1}{2};\delta_0)
&  \text{for $i=0$ }\\
q^{2s-1}q^{1/2-s}\frac{\bar{\tau}(\epsilon^i)}{\sqrt{q}} 
Z_2(1-s,w+s-\frac{1}{2};\delta_i) & \text{for $0<i<n$. }
\end{array}\right.
\]
\end{theorem}

\begin{proof}
This is a direct computation utilizing the functional
equation  (\ref{fe_incomplete}) for $L(s,\chi_{m_0})$.
\end{proof}

\section{Convolutions}

We define a convolution operation on rational functions in $x$ and $y$
with  power series expansions around the origin.  For
\[
A(x,y)=\sum_{j,k\geq 0}a(j,k)x^jy^k
\text{\ \ and\ \ }
B(x,y)=\sum_{j,k\geq 0}b(j,k)x^jy^k\textrm{,}
\]
define
\[
(A\star B)(x,y)=\sum_{j,k\geq 0}a(j,k)b(j,k)x^jy^k\textrm{.}
\]
We can compute convolutions as a double integral:
\[
(A\star B)(x,y)=
\left(\frac{1}{2\pi i}\right)^2\int\int 
A(u,v)B(\frac xu,\frac yv)\frac{du dv}{uv}
\]
where each integral is a counterclockwise circuit of a small circle
in the complex plane. (The circle must be small enough that $A(x,y)$
is holomorphic for $x,y$ inside the circle.) We will utilize the residue 
theorem to compute this contour integral.

\section{Evaluation of $Z_1$ and $Z_2$}

\label{section_eval}

We will now prove Theorem \ref{thm:Z1Z2}.  We first establish the
identity (\ref{eqn:z1rf}); then (\ref{eqn:z2rf}) will follow from the
functional equation \ref{thm:fe}.
It follows from Proposition \ref{prop:Lhatchi} that
\begin{equation}
\label{coeff_0}
\sum_{\substack{d\in\OOmon\\\deg d=k}}\chi_{m_0}(\hat{d})a(d,m)=0
\end{equation}
when $\deg m\leq k$ unless $m$ is a perfect $n\tha$ power.
To prove (\ref{eqn:z1rf}) of Theorem \ref{thm:Z1Z2}, we begin by writing 
\begin{equation}
\label{Z_decomp}
Z_1(s,w)=Z_a(s,w)+Z_a(w,s)-Z_b(s,w)
\end{equation}
where 
\[
Z_a(s,w)=\sum_{k\geq j\geq 0}\frac{1}{q^{jw}q^{ks}}
\sum_{\substack{d,m\in\OOmon\\\deg m=j\\\deg d=k}}\chi_{m_0}(\hat{d})a(d,m)
\]
and
\[
Z_b(s,w)=\sum_{k\geq 0}\frac{1}{q^{kw}q^{ks}}
\sum_{\substack{m\in\OOmon\\\deg m=j}}\sum_{\substack{d\in\OOmon\\\deg d=k}}\chi_{m_0}(\hat{d})a(d,m).
\]
Firstly, note that 
\[
\sum_{\substack{m\in\OOmon\\\deg m=j}}
\sum_{\substack{d\in\OOmon\\\deg d=k}}
\chi_{m_0}(\hat{d})a(d,m)=
\sum_{\substack{m\in\OOmon\\\deg m=j}}
\sum_{\substack{d\in\OOmon\\\deg d=k}}
\chi_{d_0}(\hat{m})a(m,d)
\]
When $m$ and $d$ are coprime, 
the reciprocity law (\ref{eq:1}) guarantees that 
$\chi_{m_0}(\hat{d})=\chi_{d_0}(\hat{m})$.  Otherwise, when $m$ and $d$ are not 
coprime, $\chi_{m_0}(\hat{d})\neq\chi_{d_0}(\hat{m})$ only when there exists a prime 
$\p$ such that $\p|d_0$ and $\p|m_0$.  In this case, $a(d,m)=0$.  The
symmetry  $a(d,m)=a(m,d)$ is obvious.  This establishes the validity of the
decomposition (\ref{Z_decomp}) of $Z_1.$

Now, the key observation is that because of equation (\ref{coeff_0}),
we have 
\[
Z_a(s,w)=\sum_{k\geq j\geq 0}\frac{1}{q^{jw}q^{ks}}
\sum_{\substack{m\in\OOmon\\\deg m=j\\ m_0=1}}
\sum_{\substack{d\in\OOmon\\\deg d=k}}a(d,m)
\]
and
\[
Z_b(s,w)=\sum_{k\geq 0}\frac{1}{q^{kw}q^{ks}}
\sum_{\substack{m\in\OOmon\\\deg m=k\\ m_0=1}}
\sum_{\substack{d\in\OOmon\\\deg d=k}}a(d,m)\textrm{,}
\]
that is, the inner sum will vanish  unless $m$ is a perfect $n\tha$-power.
This leads us to consider the series 
\[
T_a(s,w)=
\sum_{\substack{m\in\OOmon\\ m_0=1}}
\sum_{\substack{d\in\OOmon}}
\frac{a(d,m)}{|m|^w|d|^s}
\]
which has Euler product
\begin{equation}\begin{split}
T_a(s,w)
&=\prod_{\p}\sum_{j,k\geq 0}\frac{a(\p^j,\p^{nk})}{|\p|^{nk}|\p|^{j}}\\
&=\prod_{\p}\frac{1-|\p|^{-s-nw}}{(1-|\p|^{-s})(1-|\p|^{-nw})(1-|\p|^{(n-1)-ns-nw})}\\
&=\frac{\zeta(s)\zeta(nw)\zeta(ns+nw-(n-1))}{\zeta(s+nw)}\\
&=\frac{1-q^{1-s-nw}}{(1-q^{1-s})(1-q^{1-nw})(1-q^{n-ns-nw})}\\
&=\frac{1-qxy^n}{(1-qx)(1-qy^n)(1-q^{n}x^ny^n)}\\
&=\tilde T_a(x,y),
\end{split}\end{equation}
with $x=q^{-s}, y=q^{-w}$.

It is clear that 
\[
Z_a(s,w)=(\tilde{T}_a\star \tilde{K}_a)(x,y)
\]
where 
\begin{align*}
\tilde{K_a}(x,y)&=\frac{1}{(1-x)(1-xy)}\\&=\sum_{j\geq k\geq 0} x^jy^k\mbox{.}
\end{align*}
We can compute this
integral by computing residues as described in the previous section.
We find
\begin{equation}
  \label{eq:8}
Z_a(s,w)=\frac{1}{(1-q^{n+1}x^{n}y^{n})(1-qx)}\textrm{.}  
\end{equation}

We can compute $Z_b$ similarly:  let $\tilde{K}_b(x,y)=\frac{1}{1-xy}$, then 
\begin{equation}\begin{split}
Z_b(s,w)
&=(\tilde{T}_a\star \tilde{K}_b)(x,y)\\
&=\frac{1}{1-q^{n+1}x^ny^n}\textrm{.}
\end{split}\end{equation}

Putting this all together,
\begin{equation}\begin{split}
Z_1(s,w)&=
\frac{1}{(1-q^{n+1}x^{n}y^{n})(1-qx)}+
\frac{1}{(1-q^{n+1}x^{n}y^{n})(1-qy)}\\&\quad\quad-\frac{1}{1-q^{n+1}x^{n}y^{n}}\\
&=\frac{1-qy+1-qx-(1-qy)(1-qx)}{(1-q^{n+1}x^{n}y^{n})(1-qx)(1-qy)}\\
&=\frac{1-q^2xy}{(1-q^{n+1}x^{n}y^{n})(1-qx)(1-qy)}\textrm{.}
\end{split}\end{equation}
This establishes (\ref{eqn:z1rf}).

With the rational function for $Z_1(s,w)$, we can use the functional equations relating 
$Z_1(s,w;\delta_i)$ and $Z_2(s,w;\delta_i)$ for $0\leq i<n$ to evaluate $Z_2(s,w)$.  
Expanding the geometric series $\frac{1}{1-qy}$ and collecting terms with the exponent on $y$ 
congruent to $i$ (mod $n$), we arrive at 
\[
Z_1(s,w;\delta_i)=\begin{cases}
\frac{1-q^{n+1}xy^n}{(1-qx)(1-q^ny^n)(1-q^{n+1}x^{n}y^{n})}&i=0\\
\frac{(q^i-q^{i+1}x)y^i}{(1-qx)(1-q^ny^n)(1-q^{n+1}x^{n}y^{n})}&0<i<n
\end{cases}
\]

Using the functional equations relating $Z_1(s,w,\delta_i)$ and $Z_2(1-s,w+s-\frac12,\delta_i)$ and remembering that 
$\left|\frac{\tau(\epsilon^i)}{\sqrt{q}}\right|=1$, we 
see that 
\[
Z_2(s,w;\delta_i)=\begin{cases}
q^{2s-1}\frac{1-q^{-s}}{1-q^{s-1}} Z_1(1-s,w+s-\frac{1}{2};\delta_i)& i=0\\
q^{2s-1}q^{1/2-s}\frac{\tau(\epsilon^i)}{\sqrt{q}} Z_1(1-s,w+s-\frac{1}{2};\delta_i)& 0<i<n
\end{cases}
\]
With this in hand, $Z_2$ is 
\[
q^{2s-1}\frac{1-q^{-s}}{1-q^{s-1}} Z_1(1-s,w+s-\frac{1}{2};\delta_0)+\sum_{i=1}^{n-1}q^{2s-1}q^{1/2-s}\frac{\tau(\epsilon^i)}{\sqrt{q}} Z_1(1-s,w+s-\frac{1}{2};\delta_i)
\]
which, when simplified, is the rational function for $Z_2$ given in Theorem \ref{thm:Z1Z2}.

\bibliographystyle{alpha}
\bibliography{sources}

\end{document}